\documentclass[12pt,twoside]{amsart}
\usepackage{amsmath}
\usepackage{amsthm}
\usepackage{amsfonts}
\usepackage{amssymb}
\usepackage{latexsym}
\usepackage{mathrsfs}
\usepackage{amsmath}
\usepackage{amsthm}
\usepackage{amsfonts}
\usepackage{amssymb}
\usepackage{latexsym}
\usepackage{geometry}
\usepackage{dsfont}
\usepackage[dvips]{graphicx}
\usepackage[colorlinks=true,linkcolor=red,citecolor=blue]{hyperref}
\usepackage{color}
\usepackage[all]{xy}

\date{}
\allowdisplaybreaks[4] \footskip=15pt
\renewcommand{\uppercasenonmath}[1]{}

\numberwithin{equation}{section} \theoremstyle{plain}
\usepackage{graphicx,amssymb}
\usepackage[all]{xy}
\usepackage{amsmath}

\allowdisplaybreaks
\usepackage{amsthm}
\usepackage{color}

\theoremstyle{plain}

\theoremstyle{plain}
\newtheorem{theorem}{Theorem}[section]
\newtheorem{proposition}[theorem]{Proposition}
\newtheorem{lemma}[theorem]{Lemma}
\newtheorem{corollary}[theorem]{Corollary}
\newtheorem{example}[theorem]{Example}
\newtheorem*{open question}{Open Question}
\newtheorem{definition}[theorem]{Definition}

\theoremstyle{definition}

\theoremstyle{remark}
\newtheorem{remark}[theorem]{Remark}

\newcommand{\Tor}{\mbox{\rm Tor}}

\newcommand{\Prufer}{Pr\"{u}fer}
\newcommand{\A}{\mathcal{A}}

\newcommand{\Q}{\mathcal{Q}}

\newcommand{\Pj}{\mathcal{P}}

\newcommand{\Id}{\mathrm{Id}}

\def\pd{{\rm pd}}
\def\id{{\rm id}}

\def\gld{\rm gl.dim}
\def\cwd{{\rm w.gl.dim}}

\def\tor{{\rm tor}}
\def\Hom{{\rm Hom}}
\def\Ext{{\rm Ext}}
\def\Tor{{\rm Tor}}

\def\fPD{{\rm fPD}}

\def\Ker{{\rm Ker}}

\def\Im{{\rm Im}}
\def\Cok{{\rm Cok}}

\def\Nil{{\rm Nil}}
\def\Ann{{\rm Ann}}

\def\Reg{{\rm Reg}}
\def\Z{{\rm Z}}

\def\Max{{\rm Max}}

\def\T{{\rm T}}
\def\E{{\rm E}}
\def\DQ{{\rm DQ}}
\def\Spec{{\rm Spec}}
\newcommand{\m}{\frak{m}}
\newcommand{\p}{\frak{p}}
\newcommand{\q}{\frak{q}}
\def\Min{{\rm Min}}

\def\fPD{{\rm fPD}}
\def\m{{\frak m}}
\def\grade{{\rm grade}}

\begin{document}
\begin{center}
{\large  \bf On $\tau_q$-projectivity and $\tau_q$-simplicity}

\vspace{0.5cm}
Xiaolei Zhang\\
\bigskip
School of Mathematics and Statistics, Shandong University of Technology,\\
Zibo 255049, China\\
E-mail: zxlrghj@163.com\\
\end{center}

\bigskip
\centerline { \bf  Abstract}
\bigskip
\leftskip10truemm \rightskip10truemm \noindent

In this paper, we first introduce and study the notion of $\tau_q$-projective modules via strongly Lucas modules, and then investigate the  $\tau_q$-global dimension $\tau_q$-\gld$(R)$ of a ring $R$. We obtain that if $R$ is a $\tau_q$-Noetherian ring, then $\tau_q$-\gld$(R)=\tau_q$-\gld$(R[x])=$\gld$(\T(R[x]))$. Finally, we study the rings over which all modules are $\tau_q$-projective (i.e., $\tau_q$-semisimple rings). In particular, we show that a ring $R$ is a $\tau_q$-semisimple ring if and only if $\T(R[x])$ (or $\T(R)$, or $\Q_0(R)$) is a semisimple ring, if and only if $R$ is a reduced ring with $\Min(R)$ finite, if and only if  every reg-injective (or  semireg-injective, or Lucas, or strongly Lucas) module is injective.
\\
\vbox to 0.3cm{}

\leftskip0truemm \rightskip0truemm
\bigskip
\section{Introduction}

Throughout this paper, we always assume all rings are  commutative rings with identity. The small finitistic dimension $\fPD(R)$ of a ring $R$, which is defined to be  the supremum of projective dimensions of modules  with finite resolution by finitely generated projective modules, is one of most important homological dimensions that attract many algebraists.  In \cite{CFFG14}, Cahen et al. posed 44 open problems in commutative ring theory. The first one is regarding the small finitistic dimension of a commutative ring: Let $R$ be a total ring of quotients. Is \fPD$(R) = 0$? Recently, Wang et al. \cite{wzcc20} utilized $\Q$-torsion theories and Lucas modules to give a total ring of quotients $R$ with \fPD$(R)> 0$ giving a negative answer to the Problem. Furthermore, using the classifications of tilting modules,  the author in this paper et al. \cite{z-fpd} obtained total rings of quotients $R$ with \fPD$(R)=n$ for every $n\geq 0$.

For further study of $\Q$-torsion theories and Lucas modules, Zhou et al. \cite{ZDC20} introduced the notion of $q$-operations over commutative rings in terms of finitely generated semi-regular ideals.  $q$-operations are semi-star operations which are weaker than the well-known  $w$-operations (see \cite{fk16}). The authors in \cite{ZDC20} also proposed $\tau_q$-Noetherian rings and study them via module-theoretic point of view, such as $\tau_q$-analogue of the Hilbert basis theorem, Krull's principal ideal theorem, Cartan-Eilenberg-Bass theorem and Krull intersection theorem. Recently, the author in this paper et al. \cite{ZQ23} introduced and studied the notions of $\tau_q$-flat modules, $\tau_q$-VN regular rings and $\tau_q$-coherent rings. In particular, they characterized $\tau_q$-VN regular rings in terms of the total ring of quotient $\T(R[x])$ and the compactness of minimal primes; and then characterized  $\tau_q$-coheret rings in terms of $\tau_q$-flat modules.

The main motivation of this paper is to introduce the notions of $\tau_q$-projective modules, explore their homological dimensions, and characterize the $\tau_q$-semisimple rings.  As our work involves $q$-operations, we give a brief introduction on them. For more details, refer to  \cite{wzcc20,fkxs20,ZQ23,ZDC20}.

Recall that an ideal $I$ of a ring $R$ is said to be \emph{dense} if $(0:_RI):=\{r\in R\mid Ir=0\}=0$;  be \emph{regular} if $I$ contains a non-zero-divisor; and be \emph{semi-regular} if there exists a finitely generated dense sub-ideal of $I$. The set of all finitely generated semi-regular ideals of $R$ is denoted by $\Q$. The \emph{ring of finite fractions} of $R$ is defined to be:
$$\Q_0(R):=\{\alpha\in \T(R[x])\mid\ \mbox{there exists}\ I\in \Q\ \mbox{such that } I\alpha\subseteq R\},$$
Let $M$ be an $R$-module. Denote by
$\tor_{\Q}(M):=\{m\in M|Im=0$, for some $I\in \Q \}.$
An $R$-module $M$ is said to be \emph{$\Q$-torsion} (resp., \emph{$\Q$-torsion-free}) if $\tor_{\Q}(M)=M$ (resp., $\tor_{\Q}(M)=0$). 
 A $\Q$-torsion-free module $M$ is called a \emph{Lucas module} if $\Ext_R^1(R/I,M)=0$ for any $I\in \Q$, and the \emph{Lucas envelope} of $M$ is given by
\begin{center}
{\rm $M_q:=\{x\in \E_R(M)|Ix\subseteq M$, for some $I\in \Q \},$}
\end{center}
where $\E_R(M)$ is the injective envelope of $M$ as an $R$-module.
By  \cite[Theorem 2.11]{wzcc20}, $M_q=\{x\in \T(M[x])|Ix\subseteq M$, for some $I\in \Q \}.$
Obviously,  a $\Q$-torsion-free module $M$ is a Lucas module if and only if $M_q=M$. A \emph{$\DQ$ ring} $R$ is a ring for which every ideal  is a Lucas module.
By \cite[Proposition 2.2]{fkxs20} $\DQ$ rings are exactly rings with small finitistic dimensions equal to $0$.
An $R$-submodule $N$ of a  $\Q$-torsion free module $M$ is called a $q$-submodule if $N_q\cap M=N$. If an ideal $I$ of $R$ is a $q$-submodule of $R$, then $I$ is also called a $q$-ideal of $R$. A \emph{maximal $q$-ideal} is an ideal of $R$ which is maximal among the $q$-submodules of $R$. The set of all maximal $q$-ideals is denoted by $q$-$\Max(R)$. Note that $q$-$\Max(R)$ is exactly the set of all maximal non-semiregular ideals of $R$, and thus is non-empty and a subset of $\Spec(R)$ (see
\cite[Proposition 2.5, Proposition 2.7]{ZDC20}).

Let $M$ and $N$ be $R$-modules. An $R$-homomorphism $f:M\rightarrow N$ is called to be  a \emph{$\tau_q$-monomorphism} (resp., \emph{$\tau_q$-epimorphism}, \emph{$\tau_q$-isomorphism}) provided that $f_\m:M_\m\rightarrow N_\m$ is a monomorphism (resp., an epimorphism, an isomorphism) over $R_\m$ for every $\m\in q$-$\Max(R)$. By \cite[Proposition 2.7(5)]{ZDC20}, an $R$-homomorphism $f:M\rightarrow N$ is a $\tau_q$-monomorphism (resp., $\tau_q$-epimorphism, $\tau_q$-isomorphism) if  and only if $\Ker(f)$ is  (resp.,  $\Cok(f)$ is, both $\Ker(f)$ and $\Cok(f)$ are) $\Q$-torsion. A sequence  $\cdots \rightarrow A_1\rightarrow A_2\rightarrow A_3 \rightarrow \cdots $ of $R$-modules is said to be  \emph{$\tau_q$-exact} provided that $\cdots \rightarrow A_\m\rightarrow B_\m\rightarrow C_\m\rightarrow \cdots$ is  exact as  $R_\m$-modules for every $\m\in q$-$\Max(R)$.

Let $M$ be an $R$-module. Then $M$ is said to be \emph{$\tau_q$-finitely generated} provided that there exists  a $\tau_q$-exact sequence $F\rightarrow M\rightarrow 0$ with $F$ finitely generated free;  $M$ is said to be \emph{$\tau_q$-finitely presented} provided that there exists a $\tau_q$-exact sequence $ F_1\rightarrow F_0\rightarrow M\rightarrow 0$ such that $F_0$ and $F_1$ are finitely generated free modules. A ring $R$ is said to be  \emph{$\tau_q$-Noetherian} if every ideal of $R$ is  $\tau_q$-finitely generated; and is said to be \emph{$\tau_q$-coherent} provided that every $\tau_q$-finitely generated ideal of $R$ is $\tau_q$-finitely presented.

This paper is  arranged as follows. In Section 2, we first introduce the strongly Lucas modules which can be seen as Lucas modules with ``hereditary properties''(see Definition \ref{s-lucas}), connect and then differentiate them with Lucas modules (see Theorem \ref{mccoy} and Example \ref{exam-sl-l}). The $\tau_q$-projective modules are defined to be the $R$-modules which lie in the left orthogonal class of all strongly Lucas modules with respect to \Ext-functor (see Definition \ref{q-proj}).
We also consider when ``$M$ is a $\tau_q$-projective $R$-module'' is equivalent to  ``$M\otimes_R\T(R[x])$ is a projective $\T(R[x])$-module'' (see Proposition \ref{T-L} and Theorem \ref{T-L-tau}).  In Section 3, we introduce and study $\tau_q$-projective dimensions of modules and $\tau_q$-global dimensions of rings. The $\tau_q$-projective dimension $\tau_q$-\pd$_R(M)$ of an $R$-module $M$ is defined to be the  length of the shortest $\tau_q$-projective resolutions of $M$; and the $\tau_q$-global dimension $\tau_q$-\gld$(R)$ of a ring $R$ is the supremum of $\tau_q$-projective dimensions of all $R$-modules (see Definition \ref{w-phi-flat }). The Hilbert syzygy Theorem for $\tau_q$-global dimensions states that if $R$ is a $\tau_q$-Noetherian ring, then $\tau_q$-\gld$(R)=\tau_q$-\gld$(R[x])=$\gld$(\T(R[x]))$  (see Theorem \ref{Hilb}). In Section 4, we mainly characterize $\tau_q$-semisimple rings (i.e. rings $R$ with  $\tau_q$-\gld$(R)=0$). In details, we show that a ring $R$ is $\tau_q$-semisimple if and only if $\T(R[x])$  (or $\T(R)$, or $\Q_0(R)$) is a semisimple ring, if and only if
$R$ is a reduced ring with finite minimal primes, if and only if every  reg-injective (or  semireg-injective, or Lucas, or strongly Lucas) module is injective (see Theorem \ref{0-d}). Besides, the $\tau_q$-simplicity of formal power series rings and amalgamation rings are considered (see Proposition \ref{pow} and Proposition \ref{am}).

For a ring $R$, we always denote by $R[x]$ (resp., $R[\![x]\!]$) the ring of polynomials (resp., formal power series) with coefficients in $R$, $\Z(R)$ the set of all zero-divisors in $R$, $\T(R)$ the total quotient ring of $R$, $\Spec(R)$ the set of all  primes of $R$, $\Min(R)$ the set of all minimal primes of $R$, $\Nil(R)$ the nil radical of $R$, \gld$(R)$ the global dimension of $R$ and $\fPD(R)$ the small finitistic dimension of $R$. Let $I$ be an ideal of $R$. Denote by $V(I)$ the primes that contains $I$. Let $M$ be an $R$-module. Denote by \pd$_R(M)$ and \id$_R(M)$ the projective dimension and injective dimension of $M$, respectively.

\section{strongly Lucas modules and $\tau_q$-projective modules}

Before introducing the notion of $\tau_q$-projective modules, we first introduce the concept of strongly Lucas modules. Recall from \cite{wzcc20} that  a $\Q$-torsion-free $R$-module $M$ is said to be a \emph{Lucas module} if $\Ext_R^1(R/I,M)=0$ for every $I\in \Q$. Now we introduce the notion of  strongly Lucas modules.

\begin{definition}\label{s-lucas}
Let $M$ be a $\Q$-torsion-free $R$-module. Then $M$ is said to be a strongly Lucas module if $\Ext_R^n(R/I,M)=0$ for every $I\in \Q$ and every $n\geq 1$.
\end{definition}

\begin{proposition}\label{c-s-lucas}
Let $M$ be a $\Q$-torsion-free $R$-module. Then
\begin{enumerate}
    \item $M$ is  a strongly Lucas module.
        \item Every $n$-th co-syzygies $\Omega_{n}(M)$ of $M$ is  a  Lucas module for every $n\geq 1$.
     \item  If $\Ext_R^n(T,M)=0$ for every $\Q$-torsion  $R$-module $T$ and $n\geq 1$.
     \end{enumerate}
\end{proposition}
\begin{proof}  $(1)\Leftrightarrow (2)$  Follows by their definitions.
$(2)\Leftrightarrow (3)$ Follows from that a $\Q$-torsion-free $R$-module $N$ is  a  Lucas module if and only if $\Ext_R^1(T,N)=0$ for every $\Q$-torsion  $R$-module $T$ (see \cite[Proposition 1.2]{sroL}).
\end{proof}

It was proved in \cite[Proposition 2.4]{ZQ23} that every  $\T(R[x])$-module is a  Lucas $R$-module. Moreover, we have the following result.

\begin{proposition}\label{naga-lucas}
Let $M$ be a $\T(R[x])$-module. Then $M$ is a strongly Lucas $R$-module.
\end{proposition}

\begin{proof} Consider the injective resolution of $M$ over $\T(R[x])$:
$$0\rightarrow M\rightarrow E_0\xrightarrow{d_0} E_1\xrightarrow{d_0}\cdots \xrightarrow{d_{n-2}}E_{n-1}\xrightarrow{d_{n-1}} E_n\xrightarrow{d_n}\cdots,$$ where each $E_i$ is an injective $\T(R[x])$-module, and thus an injective $R$-module by \cite[Exercise 3.15]{fk16}.
Let $I\in \Q$ and $n\geq 1$. Then $\Ext_R^n(R/I,M)\cong \Ext_R^1(R/I,\Ker(d_{n-1}))$. Note that $\Ker(d_{n-1})$ is a $\T(R[x])$-module, and so a  Lucas $R$-module by  \cite[Proposition 2.4]{ZQ23}. Therefore
 $\Ext_R^1(R/I,\Ker(d_{n-1}))=0$, and thus $\Ext_R^n(R/I,M)=0$. Consequently, $M$ is a strongly Lucas $R$-module.
\end{proof}

\begin{proposition}\label{naga-lucas-p}
Let $\p$ be a non-semiregular prime ideal of $R$ and $M$ an $R_\p$-module. Then $M$ is a strongly Lucas $R$-module.
\end{proposition}

\begin{proof} Let $I\in\Q$ and $m\in M$ satisfying $Im=0$. Since $I\not\subseteq\p$, there exists $s\in I-\p$. Since $sm\in Im=0$ and $s$ is a unit in $R_\p$, we have $m=0$ implying $M$ is $\Q$-torsion-free. Next, we will show $M$ is a Lucas $R$-module. Indeed, let $I\in\Q$ and $e\in E(M)$ with $Ie\subseteq M$. Let $s\in I-\p$. Then $se\in M$, and so $e\in M$ as $s$ is a unit in $R_\p$. Thus $M$ is a Lucas $R$-module.

The rest is similar with the proof of Proposition \ref{naga-lucas}, but we exhibit it for completeness. Consider the injective resolution of $M$ over $R_\p$:
$$0\rightarrow M\rightarrow E_0\xrightarrow{d_0} E_1\xrightarrow{d_0}\cdots \xrightarrow{d_{n-2}}E_{n-1}\xrightarrow{d_{n-1}} E_n\xrightarrow{d_n}\cdots,$$ where each $E_i$ is an injective $R_\p$-module, and thus an injective $R$-module by \cite[Exercise 3.15]{fk16}.
Let $I\in \Q$ and $n\geq 1$. Then $\Ext_R^n(R/I,M)\cong \Ext_R^1(R/I,\Ker(d_{n-1}))$. Since $\Ker(d_{n-1})$ is an $R_\p$-module, and then a  Lucas $R$-module by the above. So $ \Ext_R^1(R/I,\Ker(d_{n-1}))=0$, and thus $\Ext_R^n(R/I,M)=0$. Consequently, $M$ is a strongly Lucas $R$-module.
\end{proof}

Recall from \cite{H88} that a ring $R$ has \emph{property $\A$} if each finitely generated ideal $I\subseteq \Z(R)$ has a nonzero annihilator, or equivalently,  every finitely generated semiregular ring is regular. A ring with property $\A$ is said to be an \emph{$\A$-ring} throughout the paper. Integral domains, Noetherian rings, nontrivial graded rings (e.g. polynomial rings), rings with Krull dimension equal to $0$ and  Kasch  rings (i.e., rings with only one dense ideal) are all $\A$-rings. By \cite[Corollary 2.6]{H88}, a ring $R$ is an $\A$-ring if and only if so is $\T(R)$. Trivially, every strongly Lucas module is a Lucas module.
The following result shows that  Lucas modules are  also strongly Lucas modules over $\A$-rings.
\begin{theorem}\label{mccoy}
Let $R$ be an $\A$-ring. Then every Lucas $R$-module is a strongly Lucas module.
\end{theorem}
\begin{proof}Let $M$ be a Lucas $R$-module and $I\in\Q$. Since $R$ be an $\A$-ring, there exists a regular element $a\in I$. Note that $\pd_RR/Ra\leq 1$, and so $\Ext^n_R(R/Ra,M)=0$ for every $n\geq 1$. Let $J$ be an arbitrary ideal of $R$ that contains $a$. It follows by \cite[Chapter VI,Proposition 4.1.4]{CE56} that $$ \Ext^1_{R/Ra}(R/J,\Hom_R(R/Ra,M))\cong \Ext^1_{R}(R/J,M)=0.$$ Consequently, $\Hom_R(R/Ra,M)$ is an injective
$R/Ra$-module by Baer criterion. So by \cite[Chapter VI,Proposition 4.1.4]{CE56} again, we have $$\Ext^n_{R}(R/I,M)\cong \Ext^n_{R/Ra}(R/I,\Hom_R(R/Ra,M)) =0,$$ for every $n\geq 1$. So $M$ is a strongly $R$-module.
\end{proof}

The following example shows that a Lucas module need not be a strongly Lucas module in general.
\begin{example}\label{exam-sl-l}\cite[Example 3.10]{z-fpd}\label{exa-fpd-n} Let $D=k[x_1,\dots,x_n]$ be a polynomial rings with $n\ (n\geq 2)$ variables over a field $k$.  Set $\m=\langle x_1,\dots,x_n\rangle$ be a maximal ideal of $D$ and $\Pj=\Max(R)-\{\m\}$. Define $R=D(+)B$ to be the idealization constructed prior to \cite[Theorem 11]{L93}, where $B=\bigoplus\limits_{\p\in\Pj}D/\p$.  By \cite[Theorem 11(c)]{L93}, the set of all semi-regular ideals of $R$ is $\{J(+)B\mid J$ is an ideal of $D$ and $\sqrt{J}=\m\}$. Let $J$ be an ideal of $D$ satisfying $\sqrt{J}=\m$ and set $I=J(+)B$. Then we have $\Ext_R^i(R/I,R)\cong \Ext_D^i(D/J,D)$  for every $i\geq 0$. Since $\grade(J,D)=\grade(\m,D)=K.\dim(D_{\m})=n$, we have  $\Ext_R^n(R/I,R)\not=0$ and $\Ext_R^i(R/I,R)=0$ for every $0\leq i<n$. Consequently, $R$ is a  Lucas module but not a strongly Lucas module.
\end{example}

Now, we are ready to introduce the notion of $\tau_q$-projective modules using strongly Lucas modules.
\begin{definition}\label{q-proj}
Let $M$ be an $R$-module. Then $M$ is said to be $\tau_q$-projective  if $\Ext_R^1(M,N)=0$ for every strongly Lucas module $N$.
\end{definition}

Trivially, projective modules and $\Q$-torsion modules are $\tau_q$-projective modules (see Proposition \ref{c-s-lucas}). An $R$-module  $M$ is $\tau_q$-projective  if $\Ext_R^n(M,N)=0$ for every strongly Lucas module $N$ and every $n\geq 1$.

\begin{proposition}\label{T-L-proj}
Let $R$ be a ring. Then the following statements hold.
\begin{enumerate}
    \item $\bigoplus\limits_{i\in\Gamma}M_i$ is $\tau_q$-projective  if and only if so is each $M_i$.
   \item Let $0\rightarrow A\rightarrow B\rightarrow C\rightarrow 0$ be an exact sequence. If $B$ and $C$ is $\tau_q$-projective, so is $A$; if $A$ and $C$ is $\tau_q$-projective, so is $B$.
   \item Suppose $A$ is  $\tau_q$-isomorphic to $B$. If one of $A$ and $B$ is $\tau_q$-projective, so is the other.
   \end{enumerate}
\end{proposition}
\begin{proof} We only prove (3) since the proofs of  $(1)$ and (2) are classical.

Let $f:A\rightarrow B$ be a $\tau_q$-isomorphism. Then $\Ker(f)$ and $\Cok(f)$ are  $\Q$-torsion. Consider the short exact sequences $0\rightarrow \Ker(f)\rightarrow A\rightarrow \Im(f)\rightarrow 0$ and $0\rightarrow \Im(f)\rightarrow B\rightarrow \Cok(f)\rightarrow 0$. Let $N$ be a strongly Lucas module. Then we have the following exact sequences:
$$0=\Hom_R(\Ker(f),N)\rightarrow\Ext^1_R(\Im(f),N)\rightarrow \Ext^1_R(A,N)\rightarrow \Ext^1_R(\Ker(f),N)=0$$
and
$$0=\Ext^1_R(\Cok(f),N)\rightarrow \Ext^1_R(B,N)\rightarrow \Ext^1_R(\Im(f),N)\rightarrow \Ext^{2}_R(\Cok(f),N)=0.$$
Consequently, $A$  is $\tau_q$-projective if and only if so is $\Im(f)$, if and only if so is $B$.
\end{proof}

Recall from \cite[Definition 3.7]{wzcc20} that a ring $R$ is said to be a  \emph{$\DQ$-ring} if  every ideal of $R$ is a Lucas module. It follows by \cite[Theorem 3.9]{wzcc20} and \cite[Proposition 2.2]{fkxs20} that a ring   $R$ is a $\DQ$-ring if and only if every prime (maximal) ideal of $R$ is non-semiregular over $R$, if and only if every $R$-module is a Lucas module, if and only if the only finitely  generated semiregular ideal is $R$ itself, if and only if the small finitistic dimension $\fPD(R)$ of $R$ is $0$. Now, we characterize $\DQ$-rings in terms of strongly Lucas modules and $\tau_q$-projective modules.

\begin{theorem}\label{dq}
Let $R$ be a  ring. Then the following statements are equivalent.
\begin{enumerate}
    \item $R$ is a $\DQ$-ring.
    \item Every $R$-module is a strongly Lucas module.
    \item  Every ideal of $R$ is a strongly Lucas module.
        \item  Every ideal of $R$ is  a Lucas module.
 \item   Every $\tau_q$-projective module is projective.
\end{enumerate}
\end{theorem}
\begin{proof} $(1)\Rightarrow (2)$ Suppose $R$ is a $\DQ$-ring. Then $\Q=\{R\}$, and thus every $R$-module is a strongly Lucas module.

$(2)\Rightarrow (5)$ and $(2)\Rightarrow (3)\Rightarrow (4)\Rightarrow (1)$ Trivial.

$(5)\Rightarrow (1)$ Let $I\in \Q$. Then $R/I$ is $\Q$-torsion, and so $\tau_q$-projective. Hence $R/I$ is projective. So $I$ is generated by an idempotent $e$ in $R$ by \cite[Chapter I, Proposition 1.10]{FS01}. Since $I$ is semi-regular, we have $I=R$, and so $R$ is a $\DQ$-ring.
\end{proof}

\begin{proposition}\label{T-L}
Let $M$ be a $\tau_q$-projective $R$-module. Then the following statements hold.
\begin{enumerate}
    \item $M\otimes_R\T(R[x])$ is a projective $\T(R[x])$-module.
   \item $M_\p$ is a free $R_\p$-module for every non-semiregular prime ideal $\p$ of $R$.
   \end{enumerate}
\end{proposition}
\begin{proof} (1) Suppose the $R$-module $M$ is  $\tau_q$-projective. Let $N$ be a $\T(R[x])$-module. Then, by Proposition \ref{naga-lucas}, $N$ is a strongly Lucas $R$-modules. It follows by \cite[Chapter VI,Proposition 4.1.3]{CE56} that $$\Ext_{\T(R[x])}^1(M\otimes_R\T(R[x]),N)\cong \Ext_{R}^1(M,N)=0.$$ Hence $M\otimes_R\T(R[x])$ is a projective $\T(R[x])$-module.

(2) It is similar with that of (1).
\end{proof}

\begin{remark}\label{c-T-L}
The converse of Proposition \ref{T-L}(2) does not hold in general. Let $R$ be a von Neumann regular ring but not semisimple. Then $R$ is a $\DQ$-ring. So every proper ideal is non-semiregular, and every $\tau_q$-projective $R$-module is projective. Since $R_\p$ is a field, so every non-projective $R$-module is a counterexample. It is an interesting question to ask whether the converse of Proposition \ref{T-L}(1) holds?
\end{remark}

\begin{theorem}\label{T-L-tau}
Let  $M$ be a $\tau_q$-finitely generated  $R$-module. If one of the following statements hold
 \begin{enumerate}
    \item $M\otimes_R\T(R[x])$ is a projective $\T(R[x])$-module;
      \item $M_{\m}$ is a free $R_{\m}$-module for every $\m\in q$-$\Max(R)$,
   \end{enumerate}
then $\Ext^1_R(M,N)$ is $\Q$-torsion for every strongly Lucas $R$-module $N$.
If, moreover, $R$ is an $\A$-ring, then $M$ is $\tau_q$-projective.
\end{theorem}
\begin{proof} We only prove the case $(1)$ since $(2)$ can be proved similarly.
Let  $M$ be a $\tau_q$-finitely generated  $R$-module. Then $M\otimes_R\T(R[x])$ is a finitely generated $\T(R[x])$-module by \cite[Theorem 3.3]{ZQ23}. Suppose $M\otimes_R\T(R[x])$ is a projective $\T(R[x])$-module. Let $N$ be a strongly Lucas $R$-module and $0\rightarrow A\rightarrow P\rightarrow M\rightarrow 0$ be a $\tau_q$-exact sequence with $P$ finitely generated projective. Then we have the following commutative diagram with rows exact (we denote by $\otimes_R\T(R[x]):=^T$):
$$\xymatrix@R=20pt@C=20pt{
0\ar[r]^{}&\Hom_R(M,N)^T \ar[r]^{}\ar[d]^{\lambda_{M}} &\Hom_R(P,N)^T \ar[d]^{\cong}_{\lambda_{P}}\ar[r]^{} &\Hom_R(A,N)^T \ar[d]_{\lambda_{A}}\\
0\ar[r]&\Hom_{\T(R[x])}(M^T,N^T) \ar[r]^{} &\Hom_{\T(R[x])}(P^T,N^T) \ar[r]^{} & \Hom_{\T(R[x])}(A^T,N^T)  \\ }$$
and
$$\xymatrix@R=20pt@C=20pt{
\Hom_R(P,N)^T \ar[r]^{}\ar[d]_{\cong}^{\lambda_{P}} &\Hom_R(A,N)^T \ar[d]_{\lambda_{A}}\ar[r]^{} &\Ext^1_R(M,N)^T \ar[d]_{\lambda^1_{M}}\ar[r]^{} &0\\
\Hom_{\T(R[x])}(P^T,N^T) \ar[r]^{} &\Hom_{\T(R[x])}(A^T,N^T) \ar[r]^{} & \Ext^1_{\T(R[x])}(M^T,N^T) \ar[r]^{} &0 \\ }$$
Since $N$ is $\Q$-torsion free, one can verify $\lambda_{A}$ is a monomorphism (similar to \cite[Theorem 6.7.14]{fk16}(1)). So $\lambda^1_{M}$ is also a monomorphism. Since $\Ext^1_{\T(R[x])}(M^T,N^T)=0$, $\Ext^1_R(M,N)$ is $\Q$-torsion by \cite[Proposition 2.3.]{ZQ23}.

Now assume $R$ is an $\A$-ring. Let $0\rightarrow A\rightarrow F\rightarrow M\rightarrow 0$ be a short exact sequence with with $F$ free.  
Consider the exact sequence
$$0\rightarrow \Hom_R(M,N)\rightarrow \Hom_R(F,N)\xrightarrow{h} \Hom_R(A,N)\rightarrow \Ext^1_R(M,N) \rightarrow0.$$
 To show $\Ext^1_R(M,N)$ is $\Q$-torsion-free, it is enough to show $\Im(h)$ is a Lucas module. Then it is enough to show $\Ext_R^2(R/I,\Hom_R(M,N))=0$ for every $I\in\Q$. Since $R$ is an $\A$-ring, we just need to show  $\Hom_R(M,N)$ is a Lucas module  by Theorem \ref{mccoy}. Considering the above exact sequence again, we just need to show  $\Im(h)$ is $\Q$-torsion-free, which is correct since it is a submodule of the $\Q$-torsion-free module $\Hom_R(A,N)$.  Consequently, $\Ext^1_R(M,N)=0$ for every strongly Lucas $R$-module $N$, that is, $M$ is $\tau_q$-projective.
\end{proof}

\begin{corollary}\label{p-f-2}
Let $R$ be an $\A$-ring and  $M$ a finitely generated $\tau_q$-projective $R$-module. Then $M[x]$ is a  $\tau_q$-projective $R[x]$-module.
\end{corollary}
\begin{proof} Since $M$ is a $\tau_q$-projective $R$-module, $M\otimes_{R}\T(R[x])$ is projective $\T(R[x])$-module by Theorem \ref{T-L-tau}. Since the natural embedding map $\T(R[x])\hookrightarrow \T(R[x,y])$ is flat, we have  $$M[x]\otimes_{R[x]}\T(R[x,y])\cong M\otimes_{R}\T(R[x,y]) \cong (M\otimes_{R}\T(R[x]))\otimes_{\T(R[x])}\T(R[x,y])$$ is an projective $\T(R[x,y])$-module. Hence the  finitely generated $R[x]$-module  $M[x]$ is a  $\tau_q$-projective by Theorem \ref{T-L-tau} again.
\end{proof}

Recall from \cite{ZQ23} that an $R$-module $M$ is said to be a \emph{$\tau_q$-flat module} provided that, for every $\tau_q$-monomorphism  $f: A\rightarrow B$,  $1_M\otimes f:M\otimes_RA\rightarrow M\otimes_R B$ is a $\tau_q$-monomorphism. It follows by \cite[Theorem 4.3]{ZQ23} that  an $R$-module $M$ is $\tau_q$-flat if and only if  $\Tor^R_1(M,N)$ is $\Q$-torsion for every $N$, if and only if $\Tor^R_n(M,N)$ is $\Q$-torsion  for every $N$ and every $n\geq 1$, if and only if  $M_\m$ is a flat $R_\m$-module for every $\m\in q$-$\Max(R)$, if and only if $M\otimes_R\T(R[x])$ is a flat  $\T(R[x])$-module.

\begin{proposition}\label{p-f}
Every $\tau_q$-projective $R$-module is $\tau_q$-flat.
\end{proposition}
\begin{proof} It follows by Proposition \ref{T-L-proj} and \cite[Theorem 4.3]{ZQ23}.
\end{proof}

\begin{proposition}\label{p-f} Let $R$ be an $\A$-ring. Then
every $\tau_q$-finitely presented $\tau_q$-flat $R$-module is $\tau_q$-projective.
\end{proposition}
\begin{proof}
It follows by Proposition \ref{T-L}, Theorem \ref{T-L-tau} and \cite[Theorem 4.3]{ZQ23}.
\end{proof}

\begin{proposition}\label{nil}
Let $I$ be a nonzero nil ideal of a  ring $R$. Then $I$ is not $\tau_q$-projective.
\end{proposition}
\begin{proof} On contrary, suppose $I$ is  $\tau_q$-projective. Since $I$ is nil,  $I\subseteq \m$ for every $\m\in q$-$\Max(R)$. Then, by Proposition \ref{T-L}, $I_m$ is a free principal ideal. So we may assume $R$ is local ring, and $I=Rr$ is a free principal ideal with $r$ nilpotent. Suppose the nilpotent index of $r$ is $n$. If $n>1$, $\Ann_R(r)\not=0$. Considering the exact sequence $0\rightarrow \Ann_R(r)\rightarrow R\rightarrow Rr \rightarrow 0$, we have $R=\Ann_R(r)$ since $R$ is local. So $I=0$. Going back to the general case, $I_\m=0$ for every $\m\in q$-$\Max(R)$, which implies that $I$ is $\Q$-torsion. Since $I$ is $\Q$-torsion-free,  we have $I=0$, a contradiction.  Hence $I$ is not $\tau_q$-projective.
\end{proof}

\section{The homological properties induced by $\tau_q$-projective modules}

Let $M$ be an $R$-module. If there exists an exact sequence $$0\rightarrow P_n\rightarrow \cdots\rightarrow P_1\rightarrow P_0\rightarrow M \rightarrow 0\ \ \ \ \ \ \ \ \ \ \ \ \ (\diamondsuit)$$
where each $P_i$ is $\tau_q$-projective, then we denote by $\tau_q$-\pd$_R(M)\leq n$ and say the  \emph{$\tau_q$-projective dimension} of $M$ is at most $n$.
The exact sequence $(\diamondsuit)$ is said to be  a  $\tau_q$-projective resolution of length $n$ of $M$. If such finite resolution does not exist, then we say  $\tau_q$-\pd$_R(M)=\infty$; otherwise,  define $\tau_q$-\pd$_R(M)= n$ if $n$ is the length of the shortest  $\tau_q$-projective  resolution of $M$.

Certainly, an $R$-module $M$ has  $\tau_q$-\pd$_R(M)=0$ if and only if $M$ is a $\tau_q$-projective module. For every $R$-module $M$, $\tau_q$-\pd$_R(M)\leq$\pd$_R(M)$, and ``$=$'' holds if $R$ is a $\DQ$-ring.

\begin{lemma}\label{s-p-ex}
Let $0\rightarrow A\xrightarrow{f} P\rightarrow M\rightarrow 0$ be a $\tau_q$-exact sequence with $P$  $\tau_q$-projective. Then for every strongly Lucas module $N$ and every integer $k\geq 1$, $\Ext_R^{k+1}(M,N)\cong \Ext_R^{k}(A,N)$.
\end{lemma}
\begin{proof}  Note that  $A$ is $\tau_q$-isomorphic to $\Im(f)$ and  $\Cok(f)$ is $\tau_q$-isomorphic to $M$. The result can easily proved from the long exact sequence induced by  the short exact seqeunce $0\rightarrow \Im(f)\rightarrow P\rightarrow \Cok(f)\rightarrow0$.
\end{proof}

\begin{proposition}\label{w-phi-flat d}
Let $R$ be a ring. The following statements are equivalent for an $R$-module $M$.
\begin{enumerate}
    \item  $\tau_q$-\pd$_R(M)\leq n$.
    \item $\Ext_R^{n+k}(M, N)=0$ for every strongly Lucas $R$-module $N$ and every integer $k\geq 1$.
    \item $\Ext_R^{n+1}(M, N)=0$ for every strongly Lucas $R$-module $N$.
    \item If $0 \rightarrow P_n \rightarrow \cdots \rightarrow P_1\rightarrow P_0\rightarrow M\rightarrow 0$ is an $($a $\tau_q$-$)$exact sequence, where $P_0, P_1,\dots, P_{n-1}$ are $\tau_q$-projective $R$-modules, then $P_n$ is also  $\tau_q$-projective.
    \item If $0 \rightarrow P_n \rightarrow \cdots \rightarrow P_1\rightarrow P_0\rightarrow M\rightarrow 0$ is an $($a $\tau_q$-$)$exact sequence, where $P_0, P_1,\dots, P_{n-1}$ are projective $R$-modules, then $P_n$ is $\tau_q$-projective.
 \item There exists $0 \rightarrow P_n \rightarrow \cdots \rightarrow P_1\rightarrow P_0\rightarrow M\rightarrow 0$ is a $\tau_q$-exact sequence, where $P_0, P_1,\dots, P_{n}$ are all $\tau_q$-projective.
\end{enumerate}
\end{proposition}
\begin{proof}

$(1)\Rightarrow (6)$, $(2)\Rightarrow (3)$ and $(4)\Rightarrow (5)$ Trivial.

$(6)\Rightarrow (2)$ Suppose  exists a $\tau_q$-exact sequence $0\rightarrow P_n\rightarrow \cdots\rightarrow P_1\rightarrow P_0\rightarrow M \rightarrow 0$ where each $P_i$ is $\tau_q$-projective. Let $N$ be a strongly Lucas $R$-module  and $k$ a  positive integer. Then $\Ext_R^{n+k}(M, N)\cong \Ext_R^{k}(P_n, N)=0$ by Lemma \ref{s-p-ex}.

$(3)\Rightarrow (4)$ Let $0 \rightarrow P_n \rightarrow \cdots \rightarrow P_1\rightarrow P_0\rightarrow M\rightarrow 0$ is $\tau_q$-exact sequence, where $P_0, P_1,\dots, P_{n-1}$ are $\tau_q$-projective $R$-modules. Let  $N$ be a strongly Lucas $R$-module.    Then $\Ext_R^{k}(P_n, N)\cong\Ext_R^{n+1}(M, N) =0$ by Lemma \ref{s-p-ex}. So $P_n$ is $\tau_q$-projective.

$(5)\Rightarrow (1)$ Let  $0 \rightarrow P_n \rightarrow \cdots \rightarrow P_1\rightarrow P_0\rightarrow M\rightarrow 0$ be an exact sequence with  $P_0, P_1,\dots, P_{n-1}$ projective. Then $P_n$ is $\tau_q$-projective by $(5)$ implying $\tau_q$-\pd$_R(M)\leq n$.
\end{proof}

The proof of the following result is classical, so we omit it.

\begin{proposition}
Let $R$ be a ring and $0\rightarrow A\rightarrow B\rightarrow C\rightarrow 0$ be an exact sequence of $R$-modules. Then the following statements hold.
\begin{enumerate}
    \item $\tau_q$-\pd$_R(C)\leq 1+\max\{\tau_q$-\pd$_R(A),\tau_q$-\pd$_R(B)\}$.
    \item if $\tau_q$-\pd$_R(B)<\tau_q$-\pd$_R(C)$, then $\tau_q$-\pd$_R(A)= \tau_q$-\pd$_R(C)-1\geq\tau_q$-\pd$_R(B)$.
\end{enumerate}
\end{proposition}

\begin{proposition}\label{T-L-1}
Let $R$ be a ring and $M$ an $R$-module. Then the following statements hold.
\begin{enumerate}
    \item $\tau_q$-\pd$_R(M) \geq$ \pd$_{\T(R[x])}(M\otimes_R\T(R[x]))$.
    \item $\tau_q$-\pd$_R(M) \geq$ \pd$_{R_\p}(M_\p)$ for every non-semiregular prime ideal $\p$ of $R$.
\end{enumerate}
\end{proposition}
\begin{proof}
(1) Suppose $\tau_q$-\pd$_R(M)\leq n$. Then there exists an exact sequence $0\rightarrow P_n\rightarrow \cdots\rightarrow P_1\rightarrow P_0\rightarrow M \rightarrow 0$
where each $P_i$ is $\tau_q$-projective.  By applying $-\otimes_R\T(R[x])$, we have  $0\rightarrow P_n\otimes_R\T(R[x])\rightarrow \cdots\rightarrow P_1\otimes_R\T(R[x])\rightarrow P_0\otimes_R\T(R[x])\rightarrow M\otimes_R\T(R[x]) \rightarrow 0$ is a projective resolution of $(M\otimes_R\T(R[x]))$
 by Proposition \ref{T-L}. So $\pd_{\T(R[x])}(M\otimes_R\T(R[x]))\leq n$.

(2) It is similar with $(1)$.
\end{proof}

The characterization of a commutative ring $R$ with  its total ring of quotients $\T(R)$ Noetherian showed in \cite{B18} can be restated as follows:

\begin{lemma}\label{t-Noe}\cite[Theorem 1.1]{B18} Let $R$ be a commutative ring. Then the following statements are equivalent.
\begin{enumerate}
 \item The total ring of quotients $\T(R)$ of $R$ is a  Noetherian ring.
 \item \begin{enumerate}
  \item every element $s+\Nil(R)$ with $s$ a non-zero-divisor in $R$ is a non-zero-divisor in $R/\Nil(R)$;
   \item the total ring of quotients $\T(R/\Nil(R))$ of $R/\Nil(R)$ is a  Noetherian ring;
    \item $\Nil(R)$ is nilpotent;
     \item each $\T(R/\Nil(R))$-module $\T(R/\Nil(R))\otimes_R(\Nil(R))^i/(\Nil(R))^{i+1}$ is finitely generated.
 \end{enumerate}
\end{enumerate}
\end{lemma}

It follows by \cite[Theorem 3.8]{ZDC20} that a ring $R$ is a $\tau_q$-Noetherian ring if and only if the total ring of quotients  $\T(R[x])$ of $R[x]$ is a Noetherian ring.
\begin{proposition} Let $R$ be a $\tau_q$-Noetherian ring. Then $\T(R)$  is a  Noetherian ring. Consequently, $R$ is an $\A$-ring.
\end{proposition}
\begin{proof} Suppose $R$ is a $\tau_q$-Noetherian ring. Then  $\T(R[x])$ is a Noetherian ring. So $R[x]$ satisfies the conditions $(a)$-$(d)$ in Lemma \ref{t-Noe}(2) by replacing $R$ to $R[x]$. To show $\T(R)$  is a  Noetherian ring, we just need to verify the conditions $(a)$-$(d)$ in Lemma \ref{t-Noe}(2). $(a),(c)$ and $(d)$ can easily be verified since $\Nil(R[x])=\Nil(R)[x]$ (see \cite[Exercise 1.47]{fk16}). As for $(b)$, since the total ring $\T(R[x]/\Nil(R[x]))
\cong \T((R/\Nil(R))[x])$ is a reduced Noetherian ring, then it is  semi-simple. So by the proof  Theorem \ref{0-d} $\T((R/\Nil(R)))$ is also a semi-simple, and so is Noetherian. It follows that  $\T(R)$  is a  Noetherian ring. Hence $\T(R)$ is an $\A$-ring,  and so $R$ is also an $\A$-ring by \cite[Corollary 2.6]{H88}
\end{proof}

\begin{corollary}\label{T-L-2}
Let $R$ be a $\tau_q$-Noetherian ring and $M$ a $\tau_q$-finitely generated $R$-module. Then
\begin{center}
$\tau_q$-\pd$_R(M) =$ \pd$_{\T(R[x])}(M\otimes_R\T(R[x])).$
\end{center}
\end{corollary}
\begin{proof} We just need to show $\tau_q$-\pd$_R(M) \leq$ \pd$_{\T(R[x])}(M\otimes_R\T(R[x]))$. Suppose \pd$_{\T(R[x])}(M\otimes_R\T(R[x]))=n$. Since $R$ is a $\tau_q$-Noetherian and $M$ a $\tau_q$-finitely generated $R$-module, there exists a $\tau_q$-exact sequence
$$0\rightarrow A\rightarrow P_{n-1}\rightarrow \cdots\rightarrow P_1\rightarrow P_0\rightarrow M \rightarrow 0$$ with each $P_i$ finitely generated projective and $A$ finitely generated.
 By applying $-\otimes_RT(R[x])$ to the above $\tau_q$-exact sequence, we have $A\otimes_RT(R[x])$ is projective over $T(R[x])$. Hence, by Theorem \ref{T-L-tau}, $A$ is $\tau_q$-projective. So $\tau_q$-\pd$_R(M) \leq n$ by Proposition \ref{w-phi-flat d}.
\end{proof}

We are ready to introduce the $\tau_q$-global dimension of a ring $R$ in terms of the supremum of  $\tau_q$-projective  dimensions of all $R$-modules.
\begin{definition}\label{w-phi-flat }
The $\tau_q$-global dimension of a ring $R$ is defined to be
\begin{center}
$\tau_q$-\gld$(R) = \sup\{\tau_q$-\pd$_R(M) \mid M $ is an $R$-module$\}.$
\end{center}
\end{definition}\label{def-wML}
Obviously, by definition, $\tau_q$-\gld$(R)\leq $\gld$(R)$ for all rings $R$.  Notice that if $R$ is a $\DQ$-ring, then $\tau_q$-\gld$(R)=$\gld$(R)$. The following result can be easily deduced by Proposition \ref{w-phi-flat d} and so we omit its proof.

\begin{theorem}\label{w-g-flat}
Let $R$ be a ring. The following statements are equivalent for $R$.
\begin{enumerate}
 \item $\tau_q$-\gld$(R)\leq  n$.
    \item $\tau_q$-\pd$_R(M)\leq n$ for every $R$-modules $M$.
     \item $\tau_q$-\pd$_R(M)\leq n$ for every finitely generated $R$-modules $M$.
    \item $\Ext_R^{n+k}(M, N)=0$ for  every $R$-modules $M$, every strongly Lucas $R$-module $N$ and every integer $k\geq 1$.

     \item  $\Ext_R^{n+1}(M, N)=0$ for  every $R$-modules $M$ and every strongly Lucas $R$-module $N$.
     \item  $\Ext_R^{n+1}(M, N)=0$ for  every finitely generated $R$-modules $M$ and every strongly Lucas $R$-module $N$.
    \item  $\Ext_R^{n+1}(R/I, N)=0$ for every strongly Lucas $R$-module $N$  and every ideal $I$ of $R$.

\end{enumerate}
Consequently, the $\tau_q$-global dimension of $R$ is determined by the
formulas:
\begin{align*}
\tau_q\mbox{-\gld}(R)=&\sup\{\tau_q\mbox{-\pd}_R(M)\mid M\ \mbox{is a finitely genrated}\ R\mbox{-module}\}\\
= & \sup\{\mbox{\id}_R(N)\mid N\ \mbox{is a strongly Lucas}\ R\mbox{-module}\}.\\
\end{align*}
\end{theorem}

\begin{proposition}
Let $R$ be a ring with $\tau_q$-\gld$(R)\leq 1$. Then $R$ is a reduced ring.
\end{proposition}
\begin{proof}
The result follows by Proposition \ref{nil}.
\end{proof}

\begin{proposition}\label{T-L-3}
Let $R$ be a ring. Then the following statements hold.
\begin{enumerate}
    \item $\tau_q$-\gld$(R) \geq\gld$$(\T(R[x]))$.
    \item $\tau_q$-\gld$(R) \geq \sup\{$\gld$(R_\m)\mid \m\in q$-$\Max(R)\}$.
\end{enumerate}
\end{proposition}
\begin{proof}
It can be easily deduced by Proposition \ref{T-L-1} and Theorem \ref{w-g-flat}.
\end{proof}

The statement in Proposition \ref{T-L-2}(2) can be strict. Let $R$ be a von Neumann regular ring but not semisimple.  Then $R_\m$ is a field for $\m\in q$-$\Max(R)$, and so \gld$(R_\m)=0$. However, $\tau_q$-\gld$(R)>0$ (see Theorem \ref{0-d} below). However, the equality holds for $\tau_q$-Noetherian rings:

\begin{proposition}\label{N-qgld}
Let $R$ be a $\tau_q$-Noetherian ring. Then
\begin{center}
$\tau_q$-\gld$(R)=\sup\{$\gld$(R_{\m})\mid \m\in q$-$\Max(R)\}$.
\end{center}
\end{proposition}
\begin{proof} The ``$\geq$'' follows by Proposition \ref{T-L-3}(2). Suppose   \gld$(R_{\m})\leq n$ for every $\m\in q$-$\Max(R)$. Let $M$ be a finitely generated $R$-module. Then there exists a $\tau_q$-exact sequence $0\rightarrow A\rightarrow P_{n-1}\rightarrow \cdots\rightarrow P_1\rightarrow P_0\rightarrow M \rightarrow 0$ where each $P_i$ is finitely generated $\tau_q$-projective and $A$ finitely generated. Then $A_{\m}$ is projective for every $\m\in q$-$\Max(R)$. Hence $A$ is $\tau_q$-projective by Theorem \ref{T-L-tau}.
\end{proof}

\begin{lemma}\label{q-poly} Let $M$ be an $R$-module and $\m\in q$-$\Max(R)$. Then $$M\otimes_RT(R[x])_{\m\otimes_R\T(R[x])}\cong M[x]_{\m[x]}.$$
\end{lemma}
\begin{proof}
 Let $\sum(\tau_q)$ be the set of all polynomials with contents in $\Q$. Then it is easy to verify $\sum(\tau_q)\subseteq R[X]-\m[X]$, so we have $$\T(R[x])_{\m\otimes_R\T(R[x])}\cong (R[x]_{\sum(\tau_q)})_{\m[x]_{\sum(\tau_q)}}\cong R[x]_{\m[x]}.$$
Hence $M\otimes_RT(R[x])_{\m\otimes_RT(R[x])}\cong M[x]_{\m[x]}.$
\end{proof}

The well-known Hilbert syzygy Theorem states that \gld$(R[x])=$ \gld$(R)+1$ for any ring $R$. Surprisingly, we have the following main result of this section.
\begin{theorem}\label{Hilb} $($\textbf{Hilbert syzygy Theorem for $\tau_q$-global dimensions}$)$ Let $R$ be a $\tau_q$-Noetherian ring. Then
\begin{center}
$\tau_q$-\gld$(R)=\tau_q$-\gld$(R[x])=$\gld$(\T(R[x]))$.
\end{center}
\end{theorem}
\begin{proof} First, we will show $\tau_q$-\gld$(R)=$\gld$(\T(R[x]))$. Since $R$ is a $\tau_q$-Noetherian ring, $\T(R[x])$ is a Noetherian ring by \cite[Theorem 3.8]{ZDC20}. So $\tau_q$-\gld$(R)=\tau_q$-\cwd$(R)$ by \cite[Theorem 3.9]{ZBAQ} and Proposition \ref{N-qgld}. Then it follows by \cite[Theorem 5.5]{ZBAQ} and \cite[Corollary 3.9.6]{fk16} that
\begin{center}
$\tau_q$-\gld$(R)=\tau_q$-\cwd$(R)=\cwd(T(R[x]))=$ \gld$(T(R[x]).$
\end{center}


Next, we will show $\tau_q$-\gld$(R)\leq \tau_q$-\gld$(R[x])$.
Since $R$ is a $\tau_q$-Noetherian ring, then $R[x]$ is also a $\tau_q$-Noetherian ring by \cite[Theorem 3.8]{ZDC20}.   Suppose $\tau_q$-\gld$(R[x])\leq n$. Let $M$ be a finitely generated $R$-module and  $\m$ be a maximal $q$-ideal of $R$, that is, a maximal non-(semi)regular ideal of $R$.  Then $\m[x]\in q$-$\Max(R[x])$. By Lemma \ref{q-poly}, we have (we denote by $-\otimes_R\T(R[x]):=(-)^T$) $$\pd_{T(R[x])_{\m^T}}(M^T)_{\m^T}=\pd_{R[x]_{\m[x]}}(M[x])_{\m[x]}\leq n.$$ Since $\T(R[x])$ is a Noetherian ring, so \pd$_{\T(R[x])}(M^T)\leq n$, and hence $\tau_q$-\pd$_R(M)\leq n$ by Corollary \ref{T-L-2}. Consequently, $\tau_q$-\gld$(R[x])\geq \tau_q$-\gld$(R)$.

Finally, we will show $\tau_q$-\gld$(R)\geq \tau_q$-\gld$(R[x])$. Suppose $\tau_q$-\gld$(R)\leq n$. Let $\m\in q$-$\Max(R[x])$. Then $\m$ is a maximal non-regular ideal of $R[x]$. Then $\m\cap R$ is also maximal  non-regular. So $\q := \m\cap R\in q$-$\Max(R)$ and $\m = \q[x]$.
It follows by Proposition \ref{N-qgld} and Lemma \ref{q-poly} that
\begin{align*}
\tau_q\mbox{-\gld}(R[x])=&\sup\{\mbox{\gld}(R[x]_\m)\mid \m\in q\mbox{-\Max}(R[x])\}\\
= & \sup\{\mbox{\gld}(T(R[x])_{\m\otimes_RT(R[x])})\mid \m\in q\mbox{-\Max}(R[x])\} \\
\leq & \tau_q\mbox{-\gld}(\T(R[x])).
\end{align*}
Consequently, $\tau_q$-\gld$(R)=\tau_q$-\gld$(R[x])$.
\end{proof}

\begin{remark}
We do not known that whether the condition that ``$R$ is a $\tau_q$-Noetherian ring'' in Theorem \ref{Hilb} can be removed. For further research, we propose the following open question:
 \begin{open question}
 What are the relationships among the three homological dimensions:  $\tau_q$-\gld$(R)$, $\tau_q$-\gld$(R[x])$ and \gld$(\T(R[x]))$ for a general ring $R$?
  \end{open question}
\end{remark}

\section{$\tau_q$-semisimple rings}

First we recall some notions on generalized injective modules utilizing Baer's Criterion.
Recall from \cite{fl11} that an $R$-module $M$ is said to be  \emph{reg-injective} if $\Ext_R^1(R/I,M)=0$ for every regular ideal $I$ of $R$; and $M$ is said to be  \emph{semireg-injective} if $\Ext_R^1(R/I,M)=0$ for every   semi-regular ideal $I$ of $R$.  Certainly,  Lucas module modules are exactly $\Q$-torsion-free  semireg-injective modules.  Moreover, we have the following implications:
$${\boxed{\mbox{injective modules}}}\Longrightarrow {\boxed{\mbox{semireg-injective modules}}}\Longrightarrow {\boxed{\mbox{reg injective modules}}}$$

The following $\tau_q$-semisimple rings to be  introduced will characterize when all reg-injective modules or all semireg-injective modules are injective (see Theorem \ref{0-d}).

\begin{definition}\label{s-lucas}
A ring $R$ is said to be $\tau_q$-semisimple if $\tau_q$-$\gld(R)=0$.
\end{definition}

Before going further, we recall some notions on Goldie rings and  rings with semisimple quotients.
Recall from \cite{B15}  that a ring $R$ is called a   \emph{Goldie ring} if it has a.c.c. on annihilators and does not contain infinite direct sums of nonzero ideals. Every integral domain is a Goldie ring. The famous Goldie Theorem states that a ring $R$ has its total ring of  quotients $\T(R)$ semisimple if and only if $R$ is reduced and Goldie (see \cite{G60}). Recently, Bavula \cite[Theorem 5.1]{B15} shows that a ring $R$ has its total ring of  quotients $\T(R)$ semisimple if and only if $R$ is a reduced ring with finite minimal primes,  and for each minimal prime $\p$, the integral domain $R/\p$ is a Goldie ring.  So for commutative rings, we have the following result.

\begin{lemma}\label{t-ss}\cite[Theorem 5.1]{B15}
Let $R$ be a commutative ring. Then $\T(R)$ is a semisimple ring if and only if $R$ is a reduced ring  with finite minimal primes.
\end{lemma}

The following main result of this section shows rings $R$ with $\T(R)$ semisimple are exactly $\tau_q$-semisimple rings.

\begin{theorem}\label{0-d}
Let $R$ be a ring. Then the following statements are equivalent.
\begin{enumerate}
    \item $R$ is a $\tau_q$-semisimple ring.
   \item Every $R$-module is $\tau_q$-projective.
   \item $\T(R[x])$ is a semisimple ring.

   \item  $R[x]$ is a reduced ring with finite minimal primes.
      \item   $\T(R)$ is a semisimple ring.
   \item  $R$ is a reduced ring with finite minimal primes.
   \item Every  reg-injective module is injective.
     \item Every  semireg-injective module is injective.
   \item Every  Lucas module is injective.
 \item Every strongly Lucas module is injective.
 \item $\Q_0(R)$ is a semisimple ring.
\end{enumerate}
\end{theorem}
\begin{proof} $(1)\Leftrightarrow (2)$ and $(7)\Rightarrow (8)\Rightarrow (9)\Rightarrow (10)$ Trivial.

$(10)\Rightarrow (2)$ It follows by the definition of $\tau_q$-projective modules.

$(2)\Rightarrow (3)$ Let $M$ be a $\T(R[x])$-module. Then it is a $\tau_q$-projective $R$-module. Thus $M\otimes_R\T(R[x])$ is an projective $\T(R[x])$-module by Proposition \ref{T-L}. Note that $$M\otimes_R\T(R[x])\cong M[x]\otimes_{R[x]}\T(R[x])\cong M[x]\otimes_{\T(R[x])}\T(R[x])\cong M[x]\cong \bigoplus_{i=1}^{\infty}M.$$
Consequently, $M$ be a projective $\T(R[x])$-module. So $\T(R[x])$ is a semisimple ring.

$(3)\Leftrightarrow (4)$ and $(5)\Leftrightarrow (6)$ They follow by Lemma \ref{t-ss}.

$(4)\Leftrightarrow (6)$ The equivalence of  reducibilities follows by  $\Nil(R[x])=\Nil(R)[x]$ (see \cite[Exercise 1.47]{fk16}).
For the equivalence of finiteness of minimal prime spectrums, we claim $\Min(R[x])=\{\p[x]\mid \p\in \Min(R)\}$.
Indeed, let $Q\subseteq \p[x]$ be a prime of $R[x]$ with $\p\in \Min(R)$. Then $R\cap Q$ is a prime ideal of $R$ and $R\cap Q\subseteq R\cap \p[x]=\p$. Hence $R\cap Q=\p$. Then $\p[x]=(R\cap Q)[x]\subseteq Q\subseteq \p[x]$, and hence $Q=\p[x]$. On the other hand, suppose $Q$ is a minimal prime of $R[x]$. Then $R\cap Q[x]=Q$. Let $\p=R\cap Q$. We will $\p$ is a minimal prime of $R$. Let $\q\subseteq \p$ be a  prime of $R$. Then $\q[x]\subseteq \p[x]=Q$ which is minimal. Hence $\q[x]=\p[x]$, and so $\q=\p$. Consequently, $\Min(R[x])=\{\p[x]\mid \p\in \Min(R)\}$. Hence the equivalence holds.

$(6)\Rightarrow (7)$ Suppose  $R$ is a reduced ring with finite minimal primes, say $\Min(R)=\{\p_1,\cdots,\p_n\}$. It follows by \cite[Proposition 1.1, Proposition 1.6]{M83} that the total ring of quotients $\T(R)\cong \bigoplus\limits_{i=1}^n R_{\p_i}$, where each  $R_{\p_i}$ is the quotient field of $R/{\p_i}$, and $\Reg(R)=R-\bigcup\limits_{i=1}^n\p_i$ by \cite[Proposition 1.1]{M83}.  Let $I$ be a  non-zero non-regular ideal of $R$. We may assume $I=\bigcap_{i\in \Lambda}\p_i$ for some nonempty $\Lambda\subsetneq \{1,\dots,n\}$ by \cite[Proposition 1.3]{M83}. Let $M$ be a reg-injective $R$-module.  To show $M$ is injective, we just need to show $\Ext_R^1(R/I,M)=0$. It follows by \cite[Proposition 1.6, Proposition 3.9]{M83} that there exits $j\not\in \Lambda$ and an element $s\in \p_j-\bigcup_{i\in \Lambda}\p_i $ such that $I\cap Rs=0$. We claim $I\oplus Rs$ is a regular ideal. On contrary, suppose $I\oplus Rs\subseteq \bigcup\limits_{i=1}^n\p_i$. Then there exists $k\in\{1,\dots,n\}$ such that $I\oplus Rs\in \p_k$. If $k\in \Lambda$, then $s\in \p_k\subseteq \bigcup_{i\in \Lambda}\p_i$, a contradiction. If $k\not\in \Lambda$, then $I=\bigcap_{i\in \Lambda}\p_i\subseteq  \p_k$. Hence there exists $i\in \Lambda$ such that $\p_i\subseteq \p_k$ which is also a contradiction. Consequently, $I\oplus Rs$ is a regular ideal.  Since $M$ is reg-injective, every $R$-homomorphism from $I\oplus Rs$ to $E$ can be lifted to $R$. Since $I$ is a direct summand of $I\oplus Rs$, every $R$-homomorphism from $I$ to $E$ can also be lifted to $R$. Consequently, $\Ext_R^1(R/I,M)=0$, and so $M$ is injective.

$(11)\Rightarrow (4)$ Suppose $\Q_0(R)$ is a semisimple ring. Then $\Q_0(R)$ is a reduced ring with finite minimal primes.  As $\Min(\Q_0(R))$ is naturally isomorphic to both $\Min(\Q_0(R)[x])$ and $\Min(\T(R[x]))$. We have $\T(R[x])$ is a reduced ring with finite minimal primes.

$(3)\Rightarrow (11)$ Suppose $\T(R[x])$ s a semisimple ring. Then, as $(11)\Rightarrow (4)$, we have $\Q_0(R)$ is a reduced ring with finite minimal primes. So the total quotient ring  of $\Q_0(R)$, namely, itself is a semisimple ring.
\end{proof}

\begin{remark}  From  Theorem \ref{0-d} and its proof, one can show that every finite product of integral domains and every reduced Noetherian ring are $\tau_q$-semisimple rings; and all  $\tau_q$-semisimple rings are $\A$-rings.
\end{remark}

The following results can be easily deduced by Theorem \ref{0-d}.
\begin{corollary}
A ring $R$ is a $\tau_q$-semisimple ring if and only if so is $R[x_1,\dots,x_n]$.
\end{corollary}

\begin{corollary}\label{pow}
A ring  $R$ is a $\tau_q$-semisimple ring if and only if so is $R[\![x_1,\dots,x_n]\!]$.
\end{corollary}
\begin{proof}
It follows by Theorem \ref{0-d} and \cite[Proposition 1.4(e),Proposition 1.5(b)]{MB20}.
\end{proof}

\begin{remark} Let $R$ be a ring, $\Lambda$ an arbitrary index set and $\{x_\Lambda\}$ a set of indeterminates on $R$.  The author in \cite{MB20} introduced and studied the minimal ideal structure of the polynomial ring $R[x_\Lambda]$ and the formal power series ring $R[\![x_\Lambda]\!]_i$ where $i = 1, 2, 3$ or an infinite cardinal number. Moreover, it follows by Theorem \ref{0-d} and \cite[Proposition 1.4(e),Proposition 1.5(b),Proposition 4.1]{MB20} that a ring  $R$ is a $\tau_q$-semisimple ring if and only if  so is $R[x_\Lambda]$,  if and only if so is $R[\![x_\Lambda]\!]_i$ ($i = 1, 2, 3$ or an infinite cardinal number).
\end{remark}

Let $f:A\rightarrow B$ be a ring homomorphism and $J$ an ideal of $B$. Following from \cite{df09} the  \emph{amalgamation} of $A$ with $B$ along $J$ with respect to $f$, denoted by $A\bowtie^fJ$, is defined as $$A\bowtie^fJ=\{(a,f(a)+j)|a\in A,j\in J\},$$  which is  a subring of of $A \times B$.  Following from \cite[Proposition 4.2]{df09}, $A\bowtie^fJ$  is the pullback $\widehat{f}\times_{B/J}\pi$,
 where $\pi:B\rightarrow B/J$ is the natural epimorphism and $\widehat{f}=\pi\circ f$:
$$\xymatrix@R=20pt@C=25pt{
A\bowtie^fJ\ar[d]^{p_B}\ar[r]_{p_A}& A\ar[d]^{\widehat{f}}\\
B\ar[r]^{\pi}&B/J. \\
}$$

\begin{lemma}\cite[Proposition 5.4]{df09}\label{am-1}
Let $f:A\rightarrow B$ be a ring homomorphism and $J$ an ideal of $B$. Then $A\bowtie^fJ$ is a reduced ring if and only if $A$ is reduced and $\Nil(B)\cap J=0$.
\end{lemma}

\begin{lemma}\cite[Proposition 2.6]{df10}\label{am-2}
Let $f:A\rightarrow B$ be a ring homomorphism and $J$ an ideal of $B$. Let $\p$ be a prime ideal of $A$ and $\q$ a prime ideal of $B$. Set
 \begin{enumerate}
    \item $\p'^f:=\p\bowtie^fJ=\{(p,f(p)+j)\mid p\in\p\}$;
   \item $\overline{\q}^f:=\{(a,f(a)+j)\in \p\bowtie^fJ\mid f(a)+j\in\q\}$.
  \end{enumerate}
 Then every prime ideal of $A\bowtie^fJ$ is of the form $\p'^f$ or $\overline{\q}^f$ with $\p\in\Spec(A)$ and $\q\in\Spec(B)$.
\end{lemma}

\begin{lemma}\cite[Corollary 2.8]{df16}\label{am-3}
Let $f:A\rightarrow B$ be a ring homomorphism and $J$ an ideal of $B$.  Set
$$\mathcal{X}=\bigcup\limits_{\q\in\Spec(B)-V(J)}V(f^{-1}(\q+J)).$$
Then the following properties hold.
 \begin{enumerate}
    \item The map defined by $\q\mapsto \overline{\q}^f$ establishes a homeomorphism of $\Min(B)-V(J)$ with $\Min(A\bowtie^fJ)-V(\{0\}\times J)$.
   \item  The map defined by $\p\mapsto \p'^f$ establishes a homeomorphism of $\Min(A)-\mathcal{X}$ with $\Min(A\bowtie^fJ)-V(\{0\}\times J)$.
  \end{enumerate}
 Therefore, we have
$$\Min(A\bowtie^fJ)=\{\p'^f\mid\p\in \Min(A)-\mathcal{X}\}\cup \{\overline{\q}^f\mid\q\in \Min(B)-V(J)\}.$$
\end{lemma}

\begin{proposition}\label{am}
Let $f:A\rightarrow B$ be a ring homomorphism and $J$ an ideal of $B$. Then $A\bowtie^fJ$ is a $\tau_q$-semisimple ring if and only if $A$ is reduced, $\Nil(B)\cap J=0$, and $\{\p'^f\mid\p\in \Min(A)-\mathcal{X}\}$ and $ \{\overline{\q}^f\mid\q\in \Min(B)-V(J)\}$ are finite sets.
\end{proposition}
\begin{proof}
 It follows by Lemma \ref{am-1}, Lemma \ref{am-3} and Theorem \ref{0-d}.
\end{proof}

Recall from \cite{DF07} that, by setting  $f=\Id_A: A\rightarrow A$ to be the identity  homomorphism of $A$, we denote by $A \bowtie J:=A \bowtie^{\Id_A}J$ and call it the amalgamated algebra of $A$ along $J$. From Proposition \ref{am}, one can easily deduce the following result.

\begin{corollary}
Let $J$ be an ideal of $A$. Then $A\bowtie J$ is a $\tau_q$-semisimple ring if and only if $A$ is  a $\tau_q$-semisimple ring.
\end{corollary}


\begin{thebibliography}{99}

\bibitem{A73}  J. T. Arnold, {Power series rings over Pr\"{u}fer domains}, Pacific J. Math. \textbf{44} (1973) 1-11.

\bibitem{B03}  A. Badawi,  {On nonnil-Noetherian rings}, Communications in Algebra \textbf{31} (2003) 1669-1677.

\bibitem{B15}    V. V. Bavula, New criteria for a ring to have a semisimple left quotient ring. J. Alg. Appl. \textbf{14}(6)(2015) 1550090, 28 p.

\bibitem{B18}  V. V. Bavula, Criteria for a ring to have a left Noetherian largest left quotient ring, Algebras and Repr. Theory, 21 (2018), no. 2, 359-373.

 \bibitem{B81}  J. M. Brewer, Power Series over Commutative Rings, Lecture Notes in Pure and Applied
Mathematics, Vol. 64, Marcel Dekker, Inc., New York, 1981.


\bibitem{CFFG14}  P. J. Cahen, M. Fontana, S. Frisch, and S. Glaz, ``Open problems in commutative ring theory'', pp. 353-375 in Commutative algebra, Springer, 2014.

\bibitem{CCD96} J. T. Condo, J. Coykendall, D.E. Dobbs, Formal power series rings over zero-dimensional SFT-rings, Comm. Algebra \textbf{24}(8) (1996) 2687-2698.



\bibitem{df09} M. D'Anna, C. Finocchiaro and M. Fontana, Amalgamated algebras along an ideal, in
{Commutative Algebra and its Applications}, eds. M. Fontana, S. Kabbaj, B. Olberding,
I. Swanson (Walter de Gruyter, Berlin, 2009), pp. 155-172.

\bibitem{df10} M. D'Anna, C. A. Finocchiaro and M. Fontana,  Properties of chains of prime
ideals in amalgamated algebras along an ideal. J. Pure Applied Algebra \textbf{214} (2010) 1633-1641.

\bibitem{df16} M. D'Anna, C. A. Finocchiaro and M. Fontana, New algebraic properties of an amalgamated
algebra along an ideal, Comm. Algebra \textbf{44} (2016) 1836-1851.

\bibitem{DF07} M. D'Anna, M. Fontana, {An amalgamated duplication of a ring along an ideal: the basic
properties}, J. Algebra Appl. \textbf{6}(3) (2007) 443-459.

\bibitem{CE56} H. Cartan, S. Eilenberg, Homological algebra, Princeton University Press, Princeton, 1956.

\bibitem{F91}  C. Faith, Annihilator ideals, associated primes, and Kasch McCoy commutative
rings, Comm. Algebra \textbf{19} (1991) 1867-1892.

\bibitem{FS01} L. Fuchs, L. Salce, {Modules over non-Noetherian domains}, Providence: AMS, 2001.

\bibitem{gt}  R. Gobel,   J.  Trlifaj, {Approximations and endomorphism algebras of modules}. De Gruyter Exp. Math., vol.  {\bf 41}, Berlin: Walter de Gruyter GmbH \& Co. KG, 2012.

\bibitem{G60} A. W. Goldie, Semi-prime rings with maximum condition, Proc. London Math. Soc.
 \textbf{10}(3) (1960) 201-220.

\bibitem{H88}   J. A. Huckaba, {Commutative rings with zero divisors}, Monographs and Textbooks in Pure and Applied Mathematics, 117. Marcel Dekker, Inc., New York, 1988.



\bibitem{L93} T. Lucas, Strong Pr\"{u}fer rings and the ring of finite fractions, J. Pure Appl. Algebra \textbf{84} (1993) 59-71.

\bibitem{M83} E. Matlis, The minimal prime spectrum of a reduced ring, Illinois J. Math. \textbf{27} (3) (1983) 353-391.


\bibitem{MB20} A. Maatallah, A. Benhissi, A note on $z$-ideals and $z^o$-ideals of the formal power series rings and polynomial rings in an infinite set of indeterminates, Algebra Colloq. \textbf{27} (3) (2020) 495-508.


\bibitem{fk16} F. G. Wang, H. Kim, {Foundations of commutative rings and their modules}, Singapore: Springer, 2016.

\bibitem{fl11} F. G. Wang, J. L. Liao, {S-injective modules and S-injective envelopes}, Acta Math. Sinica (Chin. Ser.) {\bf 52}(2) (2011) 271-284. (In Chinese)

\bibitem{fq15} F. G. Wang, L. Qiao, The $w$-weak global dimension of commutative rings. Bull. Korean Math. Soc. 52 (2015), 1327-1338.

\bibitem{wzcc20}  F. G. Wang, D. C. Zhou, D. Chen, {Module-theoretic characterizations of the ring of finite fractions of a commutative ring}, J. Commut. Algebra  {\bf 14}(1) (2022)  141-154.

\bibitem{fkxs20} F. G. Wang, D. C. Zhou, H. Kim, T. Xiong,  X. W. Sun,  {Every \Prufer\ ring does not have small finitistic dimension at most one}, Comm. Algebra, {\bf 48}(12) (2020) 5311-5320.

\bibitem{ywzc11} H. Y. Yin, F. G. Wang,  X. S. Zhu,  Y. H. Chen, {$w$-modules over commutative rings}, J. Korean Math. Soc. {\bf48}(1) (2011) 207-222.

\bibitem{z-q0pvmr} X. L. Zhang, {A homological characterization of $Q_0$-Pr\"{u}fer $v$-multiplication rings}, Int. Electron. J. Algebra, {\bf32} (2022) 228-240.


\bibitem{ZBAQ}X. L. Zhang, N. Bian, R. A. K. Assaad, W. Qi, On $\tau_q$-weak global dimensions of commuative rings, Southeast Asian Bull. Math., to appear.\\ https://arxiv.org/abs/2304.09323.


\bibitem{ZQ23} X. L. Zhang, W. Qi, On $\tau_q$-flatness and $\tau_q$-coherence,   J.  Appl. Algebra,  to appear.\\
 https://arxiv.org/abs/2111.03417.

\bibitem{sroL} X. L. Zhang, W. Qi, and G. C. Dai, Some remarks on Lucas modules,\\ https://arxiv.org/abs/2206.05767.



\bibitem{z-fpd}  X. L. Zhang, F. G. Wang, {The small finitistic dimensions of commutative rings}, J. Commut. Algebra, {\bf15}(1) (2023) 131-138.



\bibitem{ZDC20}  D. C. Zhou, H. Kim, F. G. Wang, D. Chen, {A new semistar operation on a commutative ring and its applications},  Comm. Algebra {\bf48} (9) (2020) 3973-3988.
\end{thebibliography}
\end{document}